 \UseRawInputEncoding
\documentclass[12pt,reqno]{amsart}
\usepackage{fullpage}

\newtheorem{theorem}{Theorem}[section]
\newtheorem{lemma}[theorem]{Lemma}

\newtheorem{cor}[theorem]{Corollary}
\newtheorem{conj}[theorem]{Conjecture}

\usepackage{graphicx}
\usepackage{color}
\usepackage[dvipsnames]{xcolor}
\usepackage{subfigure}
\usepackage{amssymb}
\usepackage{amsmath,mathrsfs}
\usepackage{colonequals}
\usepackage{hyperref}

\theoremstyle{definition}
\newtheorem{definition}[theorem]{Definition}

\newtheorem{remark}[theorem]{Remark}

\renewcommand{\subset}{\subseteq}

\renewcommand{\epsilon}{\varepsilon}
\renewcommand{\nu}{v}

\newcommand{\abs}[1]{\left|#1\right|}                   
\newcommand{\absf}[1]{|#1|}                             
\newcommand{\vnorm}[1]{\left\|#1\right\|}    

\newcommand{\Z}{\mathbb{Z}}                             

\renewcommand{\P}{\mathbb{P}}
\newcommand{\R}{\mathbb{R}}

\newcommand{\italicize}[1]{\textit {#1}}                

\newcommand{\embolden}[1]{\textbf {#1}}

\newcommand{\pdized}{periodic\,}
\newcommand{\Pdized}{Periodic\,}

\newcommand{\ksz}{\Z} 
\newcommand{\zks}{\Z} 
\newcommand{\ksexpone}{}

\newcommand{\adimn}{n}    
\newcommand{\gdimn}{n}                         
\newcommand{\pn}{p_{\adimn}}

\begin{document}

\title{A Periodic Isoperimetric Problem Related to\\ the Unique Games Conjecture}

\author{Steven Heilman}
\address{Department of Mathematics, UCLA, Los Angeles, CA 90095-1555}
\address{Department of Mathematics, University of Notre Dame, Notre Dame, IN 46656}
\email{stevenmheilman@gmail.com}
\date{\today}
\thanks{Supported by NSF Grant DMS 1708908.}

\begin{abstract}


We prove the endpoint case of a conjecture of Khot and Moshkovitz related to the Unique Games Conjecture, less a small error.

Let $n\geq2$.  Suppose a subset $\Omega$ of $n$-dimensional Euclidean space $\mathbb{R}^{n}$ satisfies $-\Omega=\Omega^{c}$ and $\Omega+v=\Omega^{c}$ (up to measure zero sets) for every standard basis vector $v\in\mathbb{R}^{n}$.  For any $x=(x_{1},\ldots,x_{n})\in\mathbb{R}^{n}$ and for any $q\geq1$, let $\|x\|_{q}^{q}=|x_{1}|^{q}+\cdots+|x_{n}|^{q}$ and let $\gamma_{n}(x)=(2\pi)^{-n/2}e^{-\|x\|_{2}^{2}/2}$ .  For any $x\in\partial\Omega$, let $N(x)$ denote the exterior normal vector at $x$ such that $\|N(x)\|_{2}=1$.  Let $B=\{x\in\mathbb{R}^{n}\colon \sin(\pi(x_{1}+\cdots+x_{n}))\geq0\}$.  Our main result shows that $B$ has the smallest Gaussian surface area among all such subsets $\Omega$, less a small error:
$$
\int_{\partial\Omega}\gamma_{n}(x)dx\geq(1-6\cdot 10^{-9})\int_{\partial B}\gamma_{n}(x)dx+\int_{\partial\Omega}\Big(1-\frac{\|N(x)\|_{1}}{\sqrt{n}}\Big)\gamma_{n}(x)dx.
$$
In particular,
$$
\int_{\partial\Omega}\gamma_{n}(x)dx\geq(1-6\cdot 10^{-9})\int_{\partial B}\gamma_{n}(x)dx.
$$
Standard arguments extend these results to a corresponding weak inequality for noise stability.  Removing the factor $6\cdot 10^{-9}$ would prove the endpoint case of the Khot-Moshkovitz conjecture.  Lastly, we prove a Euclidean analogue of the Khot and Moshkovitz conjecture.

The full conjecture of Khot and Moshkovitz provides strong evidence for the truth of the Unique Games Conjecture, a central conjecture in theoretical computer science that is closely related to the P versus NP problem.  So, our results also provide evidence for the truth of the Unique Games Conjecture.  Nevertheless, this paper does not prove any case of the Unique Games conjecture.
\end{abstract}
\maketitle

\section{Introduction}\label{secintro}

The Unique Games Conjecture is a central unresolved problem in theoretical computer science, of similar significance to the P versus NP problem.  That is, proving or disproving the Unique Games Conjecture will have significant ramifications throughout both computer science and mathematics \cite{khot10a}.  Both positive and negative evidence has been found for the Unique Games Conjecture since its formulation in 2002 by Khot \cite{khot02}, but the Conjecture remains open.  Khot's Conjecture can be formulated as follows.

\begin{definition}[\embolden{Gap Unique Games Problem} \cite{khot02,khot07}]\label{def3p}
Let $0<s<c<1$ and let $p>1$ be a prime.  We refer to $\mathrm{GapUG}_{p}(c,s)$ as the following problem.  Suppose $a_{2},a_{4},\ldots,a_{2n}\in\Z/p\Z$ are fixed, $x_{1},\ldots,x_{k}$ are variables with $k\leq n$, $1\leq i_{1},\ldots,i_{2n}\leq k$, and we have a system of $n$ two-term linear equations in $\Z/ p\Z$ of the form $x_{i_{1}}-x_{i_{2}}=a_{2}$, $x_{i_{3}}-x_{i_{4}}=a_{4},\ldots,$ $x_{i_{2n-1}}-x_{i_{2n}}=a_{2n}$.  Let $\mathrm{OPT}$ be the maximum number of these linear equations that can be satisfied by any assignment of values to the variables $x_{i_{1}},\ldots,x_{i_{2n}}$.  Decide whether $\mathrm{OPT}\geq cn$ or $\mathrm{OPT}\leq sn$.
%
\end{definition}

As $\abs{c-s}$ increases to a value near $1$, the GapUG problem becomes easier to solve.  The Unique Games Conjecture says that, even when $\abs{c-s}$ is very close to $1$, the GapUG problem is still hard to solve.  That is, the GapUG problem is nearly as hard as one could expect.

\begin{conj}[\embolden{Unique Games Conjecture} \cite{khot02,khot07}]\label{conj0}
For any $\epsilon>0$, there exists some prime $p=p(\epsilon)$ such that $\mathrm{GapUG}_{p}(1-\epsilon,\epsilon)$ is NP-hard.
%
\end{conj}
In short, Conjecture \ref{conj0} says that approximately solving linear equations is hard.  If all $n$ of the equations could be satisfied, then classical Gaussian elimination could find values for the variables $x_{i_{1}},\ldots,x_{i_{2n}}$ satisfying all of the linear equations in polynomial time in $n$.  On the other hand, if only \italicize{almost} all of the equations can be satisfied, then it is hard to satisfy a small fraction of them, according to Conjecture \ref{conj0}.  Note also that $p$ must depend on $\epsilon$ in Conjecture \ref{conj0}, since if $p$ is fixed, then a random assignment of values to the variables will satisfy a positive fraction of the linear equations.

The most significant negative evidence for Conjecture \ref{conj0} is a subexponential time algorithm for the Unique Games Problem \cite{arora09}.  That is, there exists a constant $0<a<1$ such that, for any $\epsilon>0$, and for any prime $p>1$, there is an algorithm with runtime $\exp(pn^{\epsilon^{a}})$ such that, if $(1-\epsilon)n$ equations among $n$ two-term linear equations in $\Z/ p\Z$ of the form $x_{1}-x_{2}=a_{2}$, $x_{3}-x_{4}=a_{4},\ldots,$ $x_{2n-1}-x_{2n}=a_{2n}$ can be satisfied, then the algorithm can satisfy $1-\epsilon^{a}$ of the equations.  If the quantity $\exp(pn^{\epsilon^{a}})$ could be replaced by a polynomial in $n$, then Conjecture \ref{conj0} would be false.

A recent breakthrough of \cite{khot18}, culminating the work of \cite{khot18a,khot18b,khot18c} and \cite{barak19a}, gives significant positive evidence for Conjecture \ref{conj0}.

\begin{theorem}[{\cite[page 55]{khot18}}]\label{khotthm}
For any $\epsilon>0$, there exists some prime $p=p(\epsilon)$ such that $\mathrm{GapUG}_{p}(\frac{1}{2}-\epsilon,\epsilon)$ is NP-hard.
\end{theorem}

As discussed in \cite{khot18}, since the subexponential algorithm of \cite{arora09} solves the GapUG problem for certain parameters $0<s<c<1$ where $c$ can have values in $(0,1)$, this ``... is
a compelling evidence, in our opinion, that the known algorithmic attacks are (far) short of disproving the Unique Games Conjecture.''

Due to Theorem \ref{khotthm}, it remains to investigate the hardness of $\mathrm{GapUG}_{p}(c,s)$ where $c\geq1/2$.  In a 2015 paper, Khot and Moshkovitz \cite{khot15} show that if a certain Gaussian noise stability inequality holds, then a weaker version of the NP-hardness of $\mathrm{GapUG}_{2}(1-\epsilon,1-\Omega(\sqrt{\epsilon}))$ is true for any $0<\epsilon<1$.  We describe this noise stability inequality below in Conjecture \ref{conj1}.  The weaker version of the GapUG problem is stated in \cite[page 3]{khot15}.  Resolving this weaker conjecture would provide significant evidence for the hardness of $\mathrm{GapUG}_{p}(1-\epsilon,1-\Omega(\sqrt{\epsilon}))$ and for Conjecture \ref{conj0} itself.  (Recall that $f\colon[0,1]\to\R$ satisfies $f=\Omega(\sqrt{\epsilon})$ if $\limsup_{\epsilon\to0^{+}}\abs{f(\epsilon)/\sqrt{\epsilon}}>0$.  We change notation below so that $\Omega$ denotes a subset of Euclidean space.)

For a review of positive and negative evidence for the Unique Games Conjecture, see \cite{agarwal15} and also \cite{arora09,rag12,barak12}.  See also \cite[Corollary 5.3]{harrow17} for more recent positive evidence.

For more background on the Unique Games Conjecture and its significance, see \cite{khot10a}.

%
%
%
%
%
%
%

\begin{definition}[\embolden{Gaussian density}]
Let $\adimn$ be a positive integer.  Define the \textit{Gaussian density} so that, for any $x=(x_{1},\ldots,x_{\adimn})\in\R^{\adimn}$,
$$\gamma_{\adimn}(x)\colonequals (2\pi)^{-\adimn/2}e^{-(x_{1}^{2}+\cdots+x_{\adimn}^{2})/2}.$$
\end{definition}
Recall that a standard $\adimn$-dimensional Gaussian random vector $X$ satisfies
$$\P(X\in C)=\int_{C}d\gamma_{\adimn}(x),\qquad\forall\, C\subset\R^{\adimn}.$$

Let $f\colon\R^{\adimn}\to[0,1]$ and let $\rho\in(-1,1)$,  define the \textit{Ornstein-Uhlenbeck operator with correlation $\rho$} applied to $f$ by
\begin{equation}\label{oudef}
\begin{aligned}
T_{\rho}f(x)
&\colonequals\int_{\R^{\adimn}}f(x\rho+y\sqrt{1-\rho^{2}})d\gamma_{\adimn}(y)\\
&=(1-\rho^{2})^{-n/2}(2\pi)^{-n/2}\int_{\R^{\adimn}}f(y)e^{-\frac{\vnorm{y-\rho x}_{2}^{2}}{2(1-\rho^{2})}}dy,
\qquad\forall\,x\in\R^{\adimn}.
\end{aligned}
\end{equation}
$T_{\rho}$ is a parametrization of the Ornstein-Uhlenbeck operator.  $T_{\rho}$ is not a semigroup, but it satisfies $T_{\rho_{1}}T_{\rho_{2}}=T_{\rho_{1}\rho_{2}}$ for all $\rho_{1},\rho_{2}\in(0,1)$.  We have chosen this definition since the usual Ornstein-Uhlenbeck operator is only defined for $\rho\in[0,1]$.

\begin{definition}[\embolden{Noise Stability}]\label{noisedef}
Let $\Omega\subset\R^{\adimn}$.  Let $\rho\in(-1,1)$.  We define the \textit{noise stability} of the set $\Omega$ with correlation $\rho$ to be
$$\int_{\R^{\adimn}}1_{\Omega}(x)T_{\rho}1_{\Omega}(x)d\gamma_{\adimn}(x)
\stackrel{\eqref{oudef}}{=}(2\pi)^{-n}(1-\rho^{2})^{-n/2}\int_{\Omega}\int_{\Omega}e^{\frac{-\|x\|_{2}^{2}-\|y\|_{2}^{2}+2\rho\langle x,y\rangle}{2(1-\rho^{2})}}dxdy.$$
Equivalently, if $X,Y\in\R^{\adimn}$ are independent $n$-dimensional standard Gaussian distributed random vectors, then
$$\int_{\R^{\adimn}}1_{\Omega}(x)T_{\rho}1_{\Omega}(x)d\gamma_{\adimn}(x)=\P\left(X\in \Omega,\, \rho X+Y\sqrt{1-\rho^{2}}\in \Omega\right).$$
\end{definition}


\subsection{Khot-Moshkovitz Conjecture on the Noise Stability of Periodic Sets}

Recall that the standard basis vectors $v_{1},\ldots,v_{n}\in\R^{\adimn}$ are defined so that, for any $1\leq i\leq n$, $v_{i}$ has a $1$ entry in its $i^{th}$ coordinate, and zeros in the other coordinates.

\begin{definition}[\embolden{\Pdized Set}]\label{pddef}
We say a subset $\Omega\subset\R^{\adimn}$ is \pdized if $\Omega+v=\Omega^{c}$ for every standard basis vector $v\in\R^{\adimn}$, and $-\Omega=\Omega^{c}$ (up to changes to $\Omega$ of Lebesgue measure zero).
\end{definition}

\begin{definition}[\embolden{\Pdized Half Space}]
Let $\epsilon_{1},\ldots,\epsilon_{\adimn}\in\{-1,1\}$.  We define a \textbf{\pdized half space} to be any set $B\subset\R^{\adimn}$ of the form
$$B=\{x=(x_{1},\ldots,x_{\adimn})\in\R^{\adimn}\colon \sin(\pi(\epsilon_{1}x_{1}+\cdots+\epsilon_{\adimn}x_{\adimn}))\geq0\}.$$
\end{definition}

The following Conjecture of Khot and Moshkovitz \cite{khot15} says that \pdized half spaces are the most noise stable \pdized sets.

\begin{conj}[\cite{khot15}]\label{conj1}
Let $1/2<\rho<1$.  Let $\Omega\subset\R^{\adimn}$ be a \pdized set.  Let $B\subset\R^{\adimn}$ be a \pdized half space.  Let $X,Y\in\R^{\adimn}$ be independent standard Gaussian random vectors. Then
$$\P\left(X\in \Omega,\, \rho X+Y\sqrt{1-\rho^{2}}\in \Omega\right)
\leq\P\left(X\in B,\, \rho X+Y\sqrt{1-\rho^{2}}\in B\right).$$
\end{conj}

\begin{figure}[ht!]
\centering
\def\svgwidth{.4\textwidth}
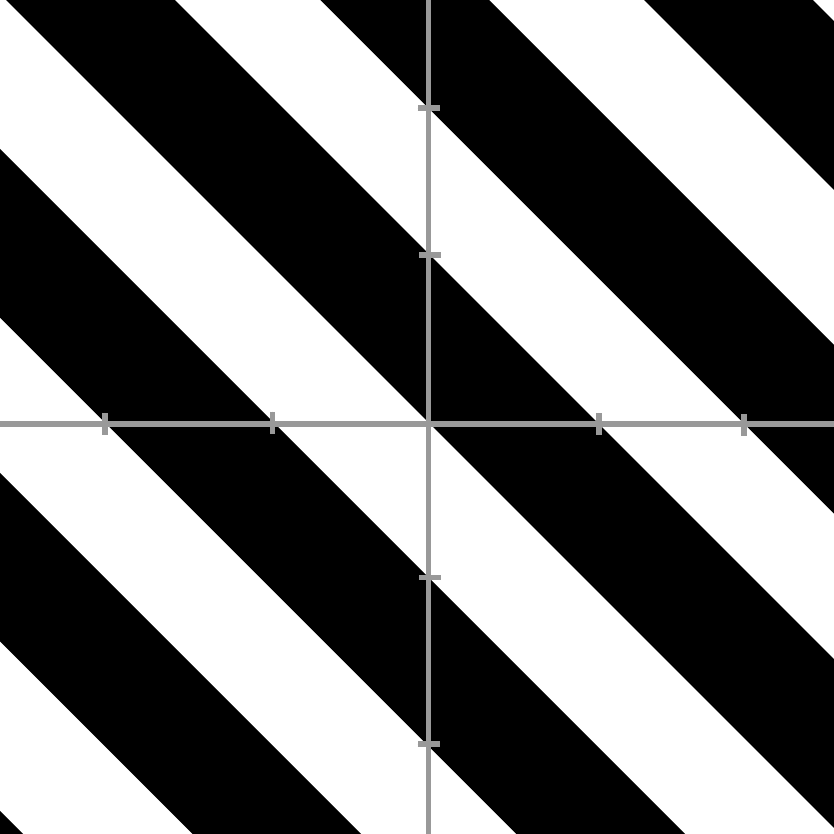
\caption{A \pdized half space $B$.}
\end{figure}

Conjecture \ref{conj1} implies that a weaker version of Conjecture \ref{conj0} holds; see \cite[page 3]{khot15} and \cite[page 5]{khot15}.  For this reason, this paper studies Conjecture \ref{conj1}.  In fact, as stated on \cite[page 5]{khot15}, a stronger version of Conjecture \ref{conj1} is required for the main application of \cite{khot15}, but we only focus on Conjecture \ref{conj1} in this work.  We are unable to prove Conjecture \ref{conj1}, so we instead study the endpoint case $\rho\to1^{-}$ of Conjecture \ref{conj1}.  As discussed in \cite{khot15}, Conjecture \ref{conj1} is most interesting and relevant to Conjecture \ref{conj0} when $\rho$ approaches $1$. That is, the case of Conjecture \ref{conj1} most relevant to the Unique Games Conjecture occurs when $\rho\to1^{-}$.

It is well known that, as $\rho\to1^{-}$, the noise stability (when normalized appropriately) converges to Gaussian surface area.  That is, if $\partial \Omega$ is a $C^{\infty}$ manifold, then \cite[Lemma 3.1]{kane11} \cite[Proposition 8.5]{ledoux96} \cite{de17}
\begin{equation}\label{zero9}
\lim_{\rho\to 1^{-}}\frac{\sqrt{2\pi}}{\cos^{-1}(\rho)}\left[\P(X\in \Omega)-\P\left(X\in \Omega,\, \rho X+Y\sqrt{1-\rho^{2}}\in \Omega\right)\right]
=\int_{\partial \Omega}\gamma_{\adimn}(x)dx.
\end{equation}
Here and below, $dx$ denotes Lebesgue measure restricted to the surface $\partial \Omega\subset\R^{\adimn}$.
Recall that a $C^{\infty}$ manifold is locally the graph of a $C^{\infty}$ function.

Letting $\rho\to1^{-}$ in Conjecture \ref{conj1} and applying \eqref{zero9} (along with $\P(X\in \Omega)=\P(X\in B)=1/2$ which follows since $-\Omega=\Omega^{c}$), we obtain the following statement.

\begin{conj}[\embolden{Endpoint $\rho\to1^{-}$ case of Conjecture \ref{conj1}}]\label{conj2}
Let $\Omega\subset\R^{\adimn}$ be a \pdized set.  Let $B\subset\R^{\adimn}$ be a \pdized half space.  Then
$$\int_{\partial \Omega}\gamma_{\gdimn}(x)dx\geq\int_{\partial B}\gamma_{\gdimn}(x)dx.$$
\end{conj}

\subsection{Our Contribution}  Our main result verifies Conjecture \ref{conj2}, up to a small error, nearly verifying the endpoint $\rho=1$ case of Conjecture \ref{conj1}, and providing evidence for the $p=2$ case of Conjecture \ref{conj0}.  Theorem \ref{thm2} also demonstrates that, if a set $\Omega$ is far from a \pdized half space, in the sense that the normal vector typically has $\ell_{1}$ norm less than $\sqrt{\adimn}$, then $\Omega$ has large Gaussian surface area.  Such a ``robustness'' statement was required in the application of \cite{khot15} to the Unique Games Conjecture, Conjecture \ref{conj0}.

\begin{theorem}[\embolden{Main Theorem; Weak Version of Conjecture \ref{conj2}}]\label{thm2}
Let $n\geq2$.  Let $\Omega\subset\R^{\adimn}$ be a \pdized set.  Let $B\subset\R^{\adimn}$ be a \pdized half space.  Assume $\partial\Omega$ is a $C^{\infty}$ manifold.  Then
\begin{equation}\label{five1}
\int_{\partial\Omega}\gamma_{\adimn}(x)dx\geq(1-6\cdot 10^{-9})\int_{\partial B}\gamma_{\adimn}(x)dx+\int_{\partial\Omega}\Big(1-\frac{\vnorm{N(x)}_{1}}{\sqrt{n}}\Big)\gamma_{\adimn}(x)dx.
\end{equation}
In particular,
\begin{equation}\label{five2}
\int_{\partial\Omega}\gamma_{\adimn}(x)dx\geq(1-6\cdot 10^{-9})\int_{\partial B}\gamma_{\adimn}(x)dx.
\end{equation}
\end{theorem}
The right-most term of \eqref{five1} is nonnegative by the Cauchy-Schwarz inequality.

Standard methods can derive from Theorem \ref{thm2} the following statement for noise stability.

\begin{cor}[\embolden{Weak Version of Khot-Moshkovitz Conjecture \ref{conj1}}]\label{cor2}
Let $d>0$.  Let $n\geq2$.  Let $g\colon\R^{\adimn}\to\R$ be a degree $d$ polynomial.  Let $\Omega=\{x\in\R^{\adimn}\colon g(x)\geq0\}$.  Let $X,Y$ be independent standard $\adimn$-dimensional Gaussian random vectors.  Let $1/2<\rho<1$.  Assume that $\Omega$ is a \pdized set and $\partial\Omega$ is a $C^{\infty}$ manifold.  Let $B\subset\R^{\adimn}$ be a \pdized half space.  Then
\begin{flalign*}
\P(X\in\Omega,\, \rho X+Y\sqrt{1-\rho^{2}}\in\Omega)
&\leq\P(X\in B,\, \rho X+Y\sqrt{1-\rho^{2}}\in B)+3\cdot 10^{-9}\\
&\qquad
-\frac{\sqrt{1-\rho^{2}}}{\sqrt{2\pi}}\int_{\partial\Omega}\Big(1-\frac{\vnorm{N(x)}_{1}}{\sqrt{n}}\Big)\gamma_{\adimn}(x)dx
+o(\sqrt{1-\rho^{2}}).
\end{flalign*}
In particular,
$$
\P(X\in\Omega,\, \rho X+Y\sqrt{1-\rho^{2}}\in\Omega)
\leq\P(X\in B,\, \rho X+Y\sqrt{1-\rho^{2}}\in B)
+3\cdot 10^{-9}+o(\sqrt{1-\rho^{2}}).
$$
Here the implied constants $o(\sqrt{1-\rho^{2}})$ can depend on $\Omega$.
\end{cor}
\begin{remark}
If $\rho$ is close to $1$ and $\int_{\partial\Omega}\gamma_{\adimn}(x)dx$ is small, then Corollary \ref{cor2} is vacuous.  In particular, if $\sqrt{1-\rho^{2}}<3\cdot 10^{-9}n^{-1/2}$, then the $3\cdot 10^{-9}$ term will be larger than the ensuing term.
\end{remark}
\begin{remark}
The ``robustness'' terms $\int_{\partial\Omega}\big(1-\frac{\vnorm{N(x)}_{1}}{\sqrt{n}}\big)\gamma_{\adimn}(x)dx$ in Theorem \ref{thm2} and Corollary \ref{cor2} can be improved slightly.  See Remark \ref{rk5} below.
\end{remark}
\begin{remark}
Under certain assumptions, the ``robustness'' term $\int_{\partial\Omega}\big(1-\frac{\vnorm{N(x)}_{1}}{\sqrt{n}}\big)\gamma_{\adimn}(x)dx$ is comparable to the Gaussian measure of the symmetric difference of $\Omega$ and a \pdized half space $B$.  See Remark \ref{rk6} below for a slightly more precise statement.
\end{remark}

If we modify the random variable $X$ in Conjecture \ref{conj1}, then we can improve Corollary \ref{cor2}.

\begin{cor}[\embolden{Modified Version of Khot-Moshkovitz Conjecture \ref{conj1}}]\label{cor4}
Let $n\geq2$.  Let $\Omega\subset\R^{\adimn}$ be a \pdized set such that $\partial\Omega$ is a $C^{\infty}$ manifold.  Let $X,Y\in\R^{\adimn}$ be independent random variables such that $X$ is uniformly distributed in $[-1/2,1/2]^{\adimn}$ and $Y$ is a standard Gaussian random vector.  Let $0<\epsilon<1/2$.  Let $B\subset\R^{\adimn}$ be a \pdized half space.  Then
$$
\P(X\in\Omega,\, X+\epsilon Y\in\Omega)
\leq\P(X\in B,\, X+\epsilon Y\in B)
-\int_{[-\frac{1}{2},\frac{1}{2}]^{\adimn}\cap \partial\Omega}\Big(1-\frac{\vnorm{N(x)}_{1}}{\sqrt{\adimn}}\Big)dx+o(\epsilon).
$$
In particular,
$$
\P(X\in\Omega,\, X+\epsilon Y\in\Omega)
\leq\P(X\in B,\, X+\epsilon Y\in B)+o(\epsilon).
$$
Here the implied constants $o(\epsilon)$ can depend on $\Omega$.
\end{cor}

%
%

\subsection{Background on Gaussian Isoperimetry}

In the 1970s, Borell and Sudakov-Tsirelson proved the Gaussian Isoperimetric Inequality \cite{borell75,sudakov74}: among all sets $\Omega\subset\R^{\adimn}$ of fixed Gaussian measure $\int_{\Omega}d\gamma_{\adimn}(x)$, the smallest Gaussian surface area $\int_{\partial \Omega}\gamma_{\gdimn}(x)dx$ occurs when $\Omega$ is a half space.  That is, $\Omega$ is the set of points lying on one side of a hyperplane.  The works \cite{borell75,sudakov74} used symmetrization methods.  That is, they replace any set $\Omega$ with a ``more symmetric'' set with the same Gaussian measure and with smaller Gaussian surface area.  In 1985, Borell generalized the Gaussian Isoperimetric Inequality to noise stability \cite{borell85}: for any $0<\rho<1$, among among all sets $\Omega\subset\R^{\adimn}$ of fixed Gaussian measure, the maximum value of $\P(X\in \Omega,\, \rho X+Y\sqrt{1-\rho^{2}}\in \Omega)$ occurs when $\Omega$ is a half space.  Once again, Borell used symmetrization methods.  Borell's result \cite{borell85} was further elucidated by many authors, including \cite{ledoux96,burchard01}.

The inequality of \cite{borell85} gained renewed attention due to its applications in theoretical computer science.  In particular, the inequality of \cite{borell85} was a key component in the proof of the Majority is Stablest Theorem \cite{mossel10}, and in the proof of the sharp Unique Games hardness of the MAX-CUT problem \cite{khot07,mossel10}.  Due to this renewed interest, Borell's result was re-proved and strengthened in \cite{mossel12,eldan13}.  The results of \cite{mossel12,eldan13} show that if a set $\Omega\subset\R^{\adimn}$ is close to maximizing the noise stability $\P(X\in \Omega,\, \rho X+Y\sqrt{1-\rho^{2}}\in \Omega)$, then $\Omega$ is close to a half space.  The work \cite{mossel12} uses heat flow methods, and \cite{eldan13} uses stochastic calculus methods.  All known proofs of Borell's inequality \cite{borell85} somehow use translation invariance of the inequality: any translation of a half space is still a half space.

Note that Conjectures \ref{conj1} and \ref{conj2} do not have any translation invariance property.  It is possible to translate a \pdized half space and produce a set that is a not a \pdized half space.  For this reason, all known proofs of Gaussian isoperimetric inequalities seem entirely unable to prove Conjectures \ref{conj1} or \ref{conj2}.

\subsection{Method of Proof of the Main Result}

Theorem \ref{thm2} is proven in an almost elementary way.  Conjectures \ref{conj1} and \ref{conj2} can be restated as isoperimetric problems on the torus equipped with the heat kernel measure on the torus.  For example, the Poisson Summation formula allows the following equivalent formulation of Conjecture \ref{conj2}.
\begin{conj}[\embolden{Restatement of Conjecture \ref{conj2}}]\label{conj2p}
The minimum value of
$$
\int_{[0,1]^{\adimn}\cap(\partial \Omega)}\sum_{z\in\Z^{\adimn}}
e^{2\pi i\langle y,z\rangle}e^{-2\pi^{2}\vnorm{z}_{2}^{2}}dy.
$$
over all \pdized sets $\Omega\subset\R^{\adimn}$ occurs when $\Omega$ is a \pdized half space.
\end{conj}
Here $\langle y,z\rangle\colonequals\sum_{i=1}^{\adimn}y_{i}z_{i}$ for any $y,z\in\R^{\adimn}$, and $\vnorm{z}_{2}^{2}=\langle z,z\rangle$.

The heat kernel measure $\pn(y)\colonequals\sum_{z\in\Z^{\adimn}}e^{2\pi i\langle y,z\rangle}e^{-2\pi^{2}\vnorm{z}_{2}^{2}}$ is very close to the constant function $1$ (see Lemma \ref{lemma53p}).  This fact may make it difficult to apply Gaussian isoperimetric methods to approach Conjecture \ref{conj2}.  So, we instead treat Conjecture \ref{conj2} as an essentially Euclidean problem.  That is, we solve exactly the analogue of Conjecture \ref{conj2p} when the integrand is the constant function $1$.  In this case, an exact solution follows by projecting $\partial B$ onto each facet of the unit cube, and noting that this projection is injective.  The error between this exact solution and the integral in Conjecture \ref{conj2p} is then small since $\pn$ is a product measure (see Lemma \ref{lemma30}).  This approach allows us to prove Theorem \ref{thm2} using an elementary argument.  The error term $6\cdot 10^{-9}$ arises since this is roughly the supremum norm of $1-p_{1}$.  That is, $6\cdot 10^{-9}$ is roughly the difference of $p_{1}$ from being constant.

\subsection{Other Related Work}


Isoperimetric problem on the torus equipped with Haar measure have been studied in several places including \cite{choksi06,ros01}, though many problems are unresolved here.  Any relation of the present work to \cite{choksi06,ros01} is unclear, since the measures under consideration are different.

Different isoperimetric problems exhibiting ``crystallization'' (or the optimality of sets consisting of
parallel stripes) have been studied in, e.g. \cite{theil06,bourne13,giuliani12,giuliani16,daneri17}, though these studies have typically focused only on $n=2$ or $n=3$.

\section{Poisson Summation Formula}

We recall some standard facts about the Poisson Summation formula.

\begin{lemma}[\embolden{Poisson Summation Formula}, {\cite[p. 252]{stein70}}]\label{lemma50}
Let $f\colon\R^{\adimn}\to\R$ be a $C^{\infty}$ function such that $\abs{f(x)}\leq 100(1+\vnorm{x}_{2})^{-2\adimn}$ for all $x\in\R^{\adimn}$.  Define $\widehat{f}(\xi)=\mathcal{F}(f)(\xi)\colonequals\int_{\R^{\adimn}}f(x)e^{-2\pi i\langle x,\xi\rangle}dx$, $\forall$ $\xi\in\R^{\adimn}$.  Let $\alpha>0$.  Then
$$\sum_{z\in(\alpha\Z)^{\adimn}}f(y+z)=\alpha^{-\adimn}\sum_{w\in(\Z/\alpha)^{n}}\widehat{f}(w)e^{2\pi i\langle y,w\rangle},\qquad\forall\,y\in\R^{\adimn}.$$
\end{lemma}

\begin{lemma}[\embolden{Eigenfunction of the Fourier Transform}, {\cite[p. 173]{stein03a}}]\label{lemma51}
$\forall$ $x\in\R$, define
$$h_{0}(x)\colonequals e^{-\pi x^{2}}.$$
Then
$$\widehat{h_{0}}(y)=h_{0}(y),\qquad\forall\,y\in\R.$$
\end{lemma}

Using the identity $\widehat{h(\cdot/\lambda)}(y)=\lambda \widehat{h}(\lambda y)$, with $\lambda=1/\sqrt{2\pi}$, we get

\begin{lemma}\label{lemma52}
For any $y\in\R$,
$$
\frac{1}{\sqrt{2\pi}}e^{- y^{2}/2}=\mathcal{F}[e^{-2\pi^{2} x^{2}}](y).
$$
\end{lemma}
%

Combining Lemmas \ref{lemma50}, \ref{lemma51} and \ref{lemma52},


\begin{lemma}\label{lemma53}
$\forall\,x\in\R$,
$$
\sum_{z\in\ksz}\gamma_{1}(x+z)
=\ksexpone\sum_{z\in\zks}e^{-2\pi^{2}z^{2}}e^{2\pi ixz}
=1+\ksexpone\sum_{k=1}^{\infty}2e^{-2\pi^{2}k^{2}}\cos(2\pi xk).
$$
\end{lemma}

\section{Weak Version of Isoperimetric Conjecture}

We denote the periodization $\pn(x)$ of the Gaussian density by
\begin{equation}\label{five0}
\pn(x)\colonequals\sum_{z\in\Z^{\adimn}}\gamma_{\adimn}(x+z),\qquad\forall\,x\in\R^{\adimn}.
\end{equation}

 We first note that $\pn$ is a product measure.  This follows directly from the definition of $\pn$.

\begin{lemma}\label{lemma30}
Let $x=(x_{1},\ldots,x_{\adimn})\in\R^{\adimn}$.  Then
$$\pn(x)=\prod_{i=1}^{\adimn}p_{1}(x_{i}).$$
\end{lemma}

We now note that $p_{1}$ is remarkably close to the constant function $1$.

\begin{lemma}\label{lemma53p}
Let $x_{1}\in\R$.  Then
$$\abs{1-p_{1}(x_{1})}\leq 54\cdot 10^{-10},\qquad\forall\, x_{1}\in\R.$$
\end{lemma}
\begin{proof}
Using Lemma \ref{lemma53} and an integral comparison,
\begin{flalign*}
\abs{p_{1}(x_{1})-1}&\stackrel{\eqref{five0}}{=}\Big|\sum_{z\in\Z}\gamma_{1}(x_{1}+z)-1\Big|
\leq2\sum_{k=1}^{\infty}e^{-2\pi^{2}k^{2}}
=2(e^{-2\pi ^{2}}+e^{-8\pi^{2}})+2\sum_{k=3}^{\infty}e^{-2\pi^{2}k^{2}}\\
&\leq 2(e^{-2\pi ^{2}}+e^{-8\pi^{2}})+2\int_{2}^{\infty}e^{-2\pi^{2}y^{2}}dy
\leq 2(e^{-2\pi ^{2}}+e^{-8\pi^{2}})+2\int_{2}^{\infty}ye^{-2\pi^{2}y^{2}}dy\\  
&= 2(e^{-2\pi ^{2}}+e^{-8\pi^{2}})+\pi^{-2}e^{-8\pi^{2}}
\leq 54\cdot 10^{-10}.
\end{flalign*}
\end{proof}

Combining Lemmas \ref{lemma30} and \ref{lemma53} proves the Main Theorem, Theorem \ref{thm2}.

\begin{proof}[Proof of Theorem \ref{thm2}]
$\forall$ $1\leq i\leq \adimn$, let $v_{i}\in\R^{\adimn}$ be the vector with a $1$ in its $i^{th}$ coordinate and a $0$ in all other coordinates.  Let $\Pi_{i}\colon[0,1]^{\adimn}\to[0,1]^{\adimn}$ be the projection onto the facet of the cube perpendicular to the $i^{th}$ coordinate, so that $\Pi_{i}(x)=x-\langle x,v_{i}\rangle v_{i}$ for all $x\in[0,1]^{\adimn}$.  Since $\Omega$ is \pdized, Definition \ref{pddef} implies that
\begin{equation}\label{five4}
\Pi_{i}([0,1]^{\adimn}\cap \partial\Omega)=\Pi_{i}([0,1]^{\adimn}),\qquad\forall\,1\leq i\leq\adimn.
\end{equation}

We first consider the case that $\partial\Omega$ consists of a finite number of flat polyhedral facets.  If $F\subset[0,1]^{\adimn}$ is one such facet, and if $N(x)$ is a unit normal vector at $x\in F$, then Lemmas \ref{lemma30} and \ref{lemma53p} together with the Cauchy projection formula (or the coordinate definition of a surface integral) imply
\begin{flalign*}
\int_{F}\abs{\langle N(x),v_{i}\rangle} \pn(x)dx
&=\int_{F}\abs{\langle N(x),v_{i}\rangle} \prod_{j=1}^{\adimn}p_{1}(x_{j})dx\\
&\geq(1-54\cdot 10^{-10})\int_{F}\abs{\langle N(x),v_{i}\rangle} \prod_{j\colon j\neq i}p_{1}(x_{j})dx\\
&=(1-54\cdot 10^{-10})\int_{\Pi_{i}(F)} \prod_{j\colon j\neq i}p_{1}(x_{j})dx.
\end{flalign*}
Summing over $1\leq i\leq \adimn$, we get
$$\int_{F}\vnorm{N(x)}_{1} \pn(x)dx
\geq(1-54\cdot 10^{-10})\sum_{i=1}^{\adimn}\int_{\Pi_{i}(F)} \prod_{j\colon j\neq i}p_{1}(x_{j})dx.$$
By approximating an arbitrary $C^{\infty}$ manifold $\partial\Omega$ by a set of flat polyhedral faces, we get
\begin{equation}\label{five4.5}
\int_{[0,1]^{\adimn}\cap \partial\Omega}\vnorm{N(x)}_{1} \pn(x)dx
\geq(1-54\cdot 10^{-10})\sum_{i=1}^{\adimn}\int_{\Pi_{i}([0,1]^{\adimn}\cap \partial\Omega)} \prod_{j\colon j\neq i}p_{1}(x_{j})dx.
\end{equation}
Then, using \eqref{five4}, we get
\begin{flalign*}
\int_{[0,1]^{\adimn}\cap \partial\Omega}\vnorm{N(x)}_{1} \pn(x)dx
&\geq(1-54\cdot 10^{-10})\sum_{i=1}^{\adimn}\int_{\Pi_{i}([0,1]^{\adimn})}\prod_{j\colon j\neq i}p_{1}(x_{j})dx\\
&=(1-54\cdot 10^{-10})\adimn\prod_{i=1}^{\adimn-1}\int_{0}^{1}p_{1}(x_{1})dx_{1}=(1-54\cdot 10^{-10})\adimn.
\end{flalign*}
In the last line, we used Lemma \ref{lemma53} (or just the definition \eqref{five0} of $p_{1}$), which implies that $\int_{0}^{1}p_{1}(x_{1})dx_{1}=1$.  Adding and subtracting the same term, we get
$$\sqrt{\adimn}\int_{[0,1]^{\adimn}\cap \partial\Omega}\pn(x)dx
+\int_{[0,1]^{\adimn}\cap \partial\Omega}(\vnorm{N(x)}_{1}-\sqrt{\adimn})\pn(x)dx\geq(1-54\cdot 10^{-10})\adimn.$$
Dividing by $\sqrt{\adimn}$ and using $\int_{\partial\Omega}\gamma_{\adimn}(x)dx
\stackrel{\eqref{five0}}{=}\int_{[0,1]^{\adimn}\cap \partial\Omega}\pn(x)dx$,
$$\int_{ \partial\Omega}\gamma_{n}(x)dx
\geq(1-54\cdot 10^{-10})\sqrt{\adimn}
+\int_{\partial\Omega}\left(1-\frac{\vnorm{N(x)}_{1}}{\sqrt{\adimn}}\right)\gamma_{n}(x)dx.$$
Then \eqref{five1} follows since $\int_{\partial B}\gamma_{\adimn}(x)dx
=\sum_{z\in (\Z/\sqrt{\adimn})}\gamma_{1}(z)
=\sqrt{\adimn}\sum_{w\in(\sqrt{\adimn}\,\Z)}e^{-2\pi^{2}w^{2}}$
by rotating $\partial B$ so that all of its hyperplanes are perpendicular to the $x_{1}$-axis, using Lemma \ref{lemma50}, and using $\adimn\geq2$.  Also, $\vnorm{N(x)}_{1}\leq\sqrt{\adimn}$ for all $x\in\partial\Omega$, so \eqref{five2} follows from \eqref{five1}.
\end{proof}
\begin{remark}\label{rk5}
The ``robustness'' term in Theorem \ref{thm2} can be improved in the following way.  Equation \eqref{five4.5} can be improved so that it counts multiple preimages of $\Pi_{i}([0,1]^{\adimn}\cap \Omega)$.  For any $1\leq i\leq \adimn$ and for any $x\in\partial\Omega$, let $\absf{\Pi_{i}^{-1}\Pi_{i}(x)}$ denote the number of preimages of $\Pi_{i}(x)$ under $\Pi_{i}$.  Then \eqref{five4.5} can be improved to
$$
\int_{[0,1]^{\adimn}\cap \partial\Omega}\sum_{i=1}^{\adimn}\frac{\abs{(N(x))_{i}}}{\absf{\Pi_{i}^{-1}\Pi_{i}(x)}} \pn(x)dx
\geq(1-54\cdot 10^{-10})\sum_{i=1}^{\adimn}\int_{\Pi_{i}([0,1]^{\adimn}\cap \partial\Omega)} \prod_{j\colon j\neq i}p_{1}(x_{j})dx.
$$
This leads to the following improvement in Theorem \ref{thm2}:
$$
\int_{\partial\Omega}\gamma_{\adimn}(x)dx\geq(1-6\cdot 10^{-9})\int_{\partial B}\gamma_{\adimn}(x)dx
+\int_{\partial\Omega}\Big(1-\sum_{i=1}^{\adimn}\frac{\abs{(N(x))_{i}}}{\sqrt{\adimn}\absf{\Pi_{i}^{-1}\Pi_{i}(x)}} \Big)\gamma_{\adimn}(x)dx.
$$
\end{remark}
\begin{remark}\label{rk6.5}
The only properties of $\pn$ used in the proof of Theorem \ref{thm2} are Lemmas \ref{lemma30} and \ref{lemma53p}, i.e. that $\pn$ is a product of one-dimensional probability measures, and each product term is close to $1$.  For example, one can replace the Gaussian measure with Lebesgue measure and prove the following similar (sharp) inequality: for any \pdized set $\Omega\subset\R^{\adimn}$,
$$\int_{[-\frac{1}{2},\frac{1}{2}]^{\adimn}\cap \partial\Omega}dx
\geq\int_{[-\frac{1}{2},\frac{1}{2}]^{\adimn}\cap \partial B}dx
+\int_{[-\frac{1}{2},\frac{1}{2}]^{\adimn}\cap \partial\Omega}\Big(1-\frac{\vnorm{N(x)}_{1}}{\sqrt{\adimn}}\Big)dx.$$
\end{remark}
\begin{remark}\label{rk6}
Let $\Omega$ be a \pdized set.  Applying the divergence theorem to the vector field $-\frac{\pn(x)}{\pi\sqrt{\adimn}}\nabla \sin(\pi(x_{1}+\cdots+x_{\adimn}))$, we get
\begin{flalign*}
&-\int_{[0,1]^{\adimn}\cap\partial\Omega}\frac{\sum_{i=1}^{\adimn}(N(x))_{i}}{\sqrt{\adimn}}\cos(\pi(x_{1}+\cdots+x_{\adimn}))\pn(x)dx\\
&=\pi\sqrt{\adimn}\int_{[0,1]^{\adimn}\cap\Omega}\left(\sin(\pi(x_{1}+\cdots+x_{\adimn}))-\Big(\frac{\sum_{i=1}^{\adimn}\frac{\partial}{\partial x_{i}}\pn(x)}{\sqrt{\adimn}\,\pn(x)}\Big)\cos(\pi(x_{1}+\cdots+x_{\adimn}))\right)\pn(x)dx.
\end{flalign*}
The integral of the $\cos$ term is bounded in absolute value by $10^{-7}\sqrt{\adimn}$ by Lemmas \ref{lemma30} and \ref{lemma53p}.   In the case that $\Omega=\{x\in\R^{\adimn}\colon \sin(\pi(x_{1}+\cdots+x_{\adimn}))\geq0\}$, we then get
$$
\abs{\int_{[0,1]^{\adimn}\cap\partial\Omega}\frac{\vnorm{N(x)}_{1}}{\sqrt{\adimn}}\gamma_{\adimn}(x)dx
-\pi\sqrt{\adimn}\int_{[0,1]^{\adimn}\cap\Omega}\sin(\pi(x_{1}+\cdots+x_{\adimn}))\gamma_{\adimn}(x)dx}
\leq 10^{-7}\sqrt{\adimn}.
$$
So, if $\Omega$ is close to a \pdized half space, then the ``robustness'' term $\int_{\partial\Omega}\big(1-\frac{\vnorm{N(x)}_{1}}{\sqrt{n}}\big)\gamma_{\adimn}(x)dx$ in Theorem \ref{thm2} and Corollary \ref{cor2} measures how close $\Omega$ is to a \pdized half space.
\end{remark}

\section{Weak Versions of Noise Stability Conjecture}

Theorem \ref{thm2} implies a similar statement for noise stability from Definition \ref{noisedef} with parameters $0<\rho<1$ that are close to $1$, using routine methods, if some smoothness is assumed for the boundary of $\Omega\subset\R^{\adimn}$.

In this section, when $\Omega\subset\R^{\adimn}$, we denote
$$f\colonequals 1_{\Omega}\colon\R^{\adimn}\to\R.$$

\begin{lemma}[{\cite[Proof of Lemma 3.1]{kane11}}]\label{lemma27}
Let $d>0$.  Let $g\colon\R^{\adimn}\to\R$ be a degree $d$ polynomial.  Let $\Omega=\{x\in\R^{\adimn}\colon g(x)\geq0\}$.  Let $X,Y\in\R^{\adimn}$ be independent standard Gaussian random vectors.  Let $\epsilon>0$.  Then
$$\P(f(X)=1,\, f(X+\epsilon Y)=0)=\frac{\epsilon}{\sqrt{2\pi}}\int_{\partial\Omega}\gamma_{\adimn}(x)dx+o(\epsilon).$$
\end{lemma}
Here and below, the implied constant $o(\epsilon)$ can depend on $g$.

\begin{lemma}[{\cite[Lemma 3.4]{kane11}}]\label{lemma28}
Let $d>0$.  Let $g\colon\R^{\adimn}\to\R$ be a degree $d$ polynomial.  Let $\Omega=\{x\in\R^{\adimn}\colon g(x)\geq0\}$.  Let $X,Y\in\R^{\adimn}$ be independent standard Gaussian random vectors.  Let $\epsilon>0$.  Then
$$\P(f(X)\neq f(X(1+\epsilon)))\leq d\epsilon\sqrt{\frac{\adimn}{4\pi}}.$$
\end{lemma}

\begin{lemma}\label{lemma29}
Let $d>0$.  Let $g\colon\R^{\adimn}\to\R$ be a degree $d$ polynomial.  Let $\Omega=\{x\in\R^{\adimn}\colon g(x)\geq0\}$.  Let $X,Y\in\R^{\adimn}$ be independent standard Gaussian random vectors.  Let $0<\eta<1/2$.  Then
$$\P(X\in\Omega,\, X\sqrt{1-\eta^{2}}+\eta Y\in\Omega)=\gamma_{\adimn}(\Omega) -\frac{\eta}{\sqrt{2\pi}}\int_{\partial\Omega}\gamma_{\adimn}(x)dx+o(\eta).$$
Here and below, the implied constant $o(\eta)$ can depend on $g$.
\end{lemma}
\begin{proof}
Let $Z\colonequals X\sqrt{1-\eta^{2}}+\eta Y$, and let $r\colonequals1/\sqrt{1-\eta^{2}}$.  Using the identity $\P(A_{1})=\P(A_{1}\cap A_{2})+\P(A_{1}\cap A_{2}^{c})$ for events $A_{1},A_{2}$,
\begin{equation}\label{four4}
\begin{aligned}
\P(f(X)=1,\, f(Z)=1)&=\P\Big(f(X)=1,\, f(Z)=1,\, f(Z)=f(rZ)\Big)\\
&\qquad+\P\Big(f(X)=1,\, f(Z)=1,\, f(Z)\neq f(rZ)\Big).
\end{aligned}
\end{equation}
We apply Lemma \ref{lemma28} to the second term of \eqref{four4} with $\epsilon\colonequals(1-\eta^{2})^{-1/2}-1$ to get
$$
\P\Big(f(X)=1,\, f(Z)=1,\, f(Z)\neq f(rZ)\Big)
\leq \P\Big(f(Z)\neq f(rZ)\Big)
\leq d\eta^{2}\sqrt{\adimn}.
$$
And the first term of \eqref{four4} is equal to 
\begin{flalign*}
\P\Big(f(X)=1,\, f(rZ)=1,\, f(Z)=f(rZ)\Big)
&=\P(f(X)=1,\, f(rZ)=1\Big)\\
&\qquad-\P\Big(f(X)=1,\, f(rZ)=1,\, f(Z)\neq f(rZ)\Big).
\end{flalign*}
Using Lemma \ref{lemma28} again, the last quantity is at most $d\eta^{2}\sqrt{\adimn}$, while
\begin{flalign*}
\P(f(X)=1,\, f(rZ)=1\Big)
&=\P(f(X)=1) -\P\Big(f(X)=1,\, f(rZ)=0\Big)\\
&=\gamma_{\adimn}(\Omega) -\frac{\eta}{\sqrt{2\pi}}\int_{\partial\Omega}\gamma_{\adimn}(x)dx+o(\eta).
\end{flalign*}
In the last line we used Lemma \ref{lemma27}.  Combining the above estimates gives
\begin{flalign*}
\P(X\in\Omega,\, X\sqrt{1-\eta^{2}}+\eta Y\in\Omega)
&=\P(f(X)=1,\, f(Z)=1)\\
&=\gamma_{\adimn}(\Omega) -\frac{\eta}{\sqrt{2\pi}}\int_{\partial\Omega}\gamma_{\adimn}(x)dx+o(\eta).
\end{flalign*}
\end{proof}

\begin{proof}[Proof of Corollary \ref{cor2}]
Let $\eta\colonequals\sqrt{1-\rho^{2}}$.  From Lemma \ref{lemma29},
$$\P(X\in\Omega,\, \rho X+Y\sqrt{1-\rho^{2}}\in\Omega)\leq \gamma_{\adimn}(\Omega) -\frac{\sqrt{1-\rho^{2}}}{\sqrt{2\pi}}\int_{\partial\Omega}\gamma_{\adimn}(x)dx+o(\sqrt{1-\rho^{2}}).$$
Since $\Omega$ is \pdized, $\Omega^{c}=-\Omega$, so $\gamma_{\adimn}(\Omega)=1/2$.  Applying Theorem \ref{thm2} gives
\begin{flalign*}
&\P(X\in\Omega,\, \rho X+Y\sqrt{1-\rho^{2}}\in\Omega)\\
&\leq\frac{1}{2} -\frac{\sqrt{1-\rho^{2}}}{\sqrt{2\pi}}
\left((1-6\cdot 10^{-9})\int_{\partial B}\gamma_{\adimn}(x)dx+\int_{\partial\Omega}\Big(1-\frac{\vnorm{N(x)}_{1}}{\sqrt{n}}\Big)\gamma_{\adimn}(x)dx\right)
+o(\sqrt{1-\rho^{2}}).
\end{flalign*}
Finally, applying Lemma \ref{lemma29} to $B$ completes the proof.
\end{proof}

By repeating the proof of Lemma \ref{lemma27}, we get the following.  For completeness, we provide a proof with a dimension-dependent implied constant.

\begin{lemma}\label{lemma32}
Let $\Omega\subset\R^{\adimn}$ be a \pdized set.  Assume that $\partial\Omega$ is a $C^{\infty}$ manifold.  Let $X,Y\in\R^{\adimn}$ be independent random variables such that $X$ is uniformly distributed in $[-\frac{1}{2},\frac{1}{2}]^{\adimn}$ and $Y$ is a standard Gaussian random vectors.  Let $\epsilon>0$.  Then
$$\P(f(X)=1,\, f(X+\epsilon Y)=0)=\frac{\epsilon}{\sqrt{2\pi}}\int_{[-\frac{1}{2},\frac{1}{2}]^{\adimn}\cap \partial\Omega}dx+o(\epsilon).$$
\end{lemma}
\begin{proof}
For any $s>0,x\in\R^{\adimn}$, and for any $f\colon\R^{\adimn}\to[0,1]$, let $U_{s} f(x)\colonequals \int_{\R^{\adimn}} f(x+y\sqrt{2s})\gamma_{\adimn}(y)dy$.  Let $\Delta\colonequals\sum_{i=1}^{\adimn}\partial^{2}/\partial x_{i}^{2}$.  It is well known that $(d/ds)U_{s}f(x)=\Delta U_{s}f(x)$ for all $s>0,x\in\R^{\adimn}$.  Using the divergence theorem,
\begin{equation}\label{ten1}
\begin{aligned}
&\frac{d}{ds}\int_{[-\frac{1}{2},\frac{1}{2}]^{\adimn}\cap\Omega}U_{s}1_{\Omega^{c}}(x)dx
=\int_{[-\frac{1}{2},\frac{1}{2}]^{\adimn}\cap\Omega}\Delta U_{s}1_{\Omega^{c}}(x)dx
=\int_{[-\frac{1}{2},\frac{1}{2}]^{\adimn}\cap\Omega}\mathrm{div}(\nabla U_{s}1_{\Omega^{c}}(x))dx\\
&\qquad\qquad=\int_{[-\frac{1}{2},\frac{1}{2}]^{\adimn}\cap \partial\Omega}\langle\nabla U_{s}1_{\Omega^{c}}(x), N(x)\rangle dx
+\int_{\big(\Omega\cap \partial [-\frac{1}{2},\frac{1}{2}]^{\adimn}\big)\setminus\partial\Omega}\langle\nabla U_{s}1_{\Omega^{c}}(x), N(x)\rangle dx.
\end{aligned}
\end{equation}
Changing variables and differentiating,
$$\nabla U_{s}1_{\Omega^{c}}(x)=\frac{1}{\sqrt{2 s}}\int_{\R^{\adimn}}y 1_{\Omega^{c}}(x+y\sqrt{2s})\gamma_{\adimn}(y)dy,\qquad\forall\,x\in\R^{\adimn}.$$
Therefore, $\lim_{s\to0^{+}}2\sqrt{\pi s}\,\nabla U_{s}1_{\Omega^{c}}(x)=N(x)$ for all $x\in\partial\Omega$.  That is,
\begin{equation}\label{ten2}
\nabla U_{s}1_{\Omega^{c}}(x)=\frac{1}{2\sqrt{\pi s}}N(x)+o(s^{-1/2}),\qquad\forall\,x\in\partial\Omega.
\end{equation}
Also, $\lim_{s\to0^{+}}2\sqrt{\pi s}\,\nabla U_{s}1_{\Omega^{c}}(x)=0$ for all $x\notin\partial\Omega$.  So, using $f=1_{\Omega}$,
\begin{flalign*}
&\P(f(X)=1,\, f(X+\epsilon Y)=0)
=\int_{[-\frac{1}{2},\frac{1}{2}]^{\adimn}\cap\Omega}U_{\epsilon^{2}/2}1_{\Omega^{c}}(x)dx\\
&\quad=\int_{s=0}^{s=\epsilon^{2}/2}\frac{d}{ds}\Big(\int_{[-\frac{1}{2},\frac{1}{2}]^{\adimn}\cap\Omega}U_{s}1_{\Omega^{c}}(x)dx\Big)ds\\
&\stackrel{\eqref{ten1}\wedge\eqref{ten2}}{=}
\Big(\int_{[-\frac{1}{2},\frac{1}{2}]^{\adimn}\cap\partial\Omega}dx\Big)\int_{s=0}^{s=\epsilon^{2}/2}\Big(\frac{1}{2\sqrt{\pi s}}+o(s^{-1/2})\Big)ds
=\frac{\epsilon}{\sqrt{2\pi}}\int_{[-\frac{1}{2},\frac{1}{2}]^{\adimn}\cap\partial\Omega}dx+o(\epsilon).
\end{flalign*}

\end{proof}

\begin{proof}[Proof of Corollary \ref{cor4}]
Note that $\P(X\in\Omega)=1/2$ since $-\Omega=\Omega^{c}$.  Using Lemma \ref{lemma32},
\begin{flalign*}
\P(X\in\Omega,\, X+\epsilon Y\in\Omega)
&=\P(X\in\Omega)-\P(X\in\Omega,\, X+\epsilon Y\notin\Omega)\\
&=\frac{1}{2}-\frac{\epsilon}{\sqrt{2\pi}}\int_{[-\frac{1}{2},\frac{1}{2}]^{\adimn}\cap \partial\Omega}dx+o(\epsilon).
\end{flalign*}
So, by Remark \ref{rk6.5} and Lemma \ref{lemma32} applied to $B$,
\begin{flalign*}
\P(X\in\Omega,\, X+\epsilon Y\in\Omega)
&\leq  \frac{1}{2}-\frac{\epsilon}{\sqrt{2\pi}}\Big(\int_{[-\frac{1}{2},\frac{1}{2}]^{\adimn}\cap \partial B}
1+\Big(1-\frac{\vnorm{N(x)}_{1}}{\sqrt{\adimn}}\Big)dx\Big)+o(\epsilon)\\
&=\P(X\in B,\, X+\epsilon Y\in B)
-\int_{[-\frac{1}{2},\frac{1}{2}]^{\adimn}\cap \partial\Omega}\Big(1-\frac{\vnorm{N(x)}_{1}}{\sqrt{\adimn}}\Big)dx+o(\epsilon).
\end{flalign*}
\end{proof}

\medskip
\noindent\textbf{Acknowledgement}.  Thanks to Elchanan Mossel and Joe Neeman for helpful discussions, especially concerning the ``robustness'' term in the main results.

\bibliographystyle{amsalpha}
\bibliography{12162011}

\end{document}